\documentclass[a4paper,11pt]{article}
\usepackage{amsthm,amsfonts, indentfirst, latexsym, bm,graphicx,amsmath,amssymb,mathrsfs,verbatim,hyperref,pdfpages}
\usepackage{caption}
\usepackage{subcaption}
\usepackage{color}
\usepackage[T1]{fontenc}
\usepackage{dsfont}
\usepackage{latexsym}
\usepackage{url}
\usepackage{array}
\usepackage{tikz}
\usepackage{tkz-berge}
\usepackage{tkz-graph}
\usepackage{pgfplots}
\usepackage{graphicx}
\usepackage{layout}
\usepackage[titletoc,title]{appendix}% this is for appendix
\usepackage{bbm} % this is for indicator function
\usepackage{algorithm,algorithmic}
\numberwithin{equation}{section}
\newtheorem{theorem}{Theorem}[section]
\newtheorem{corollary}[theorem]{Corollary}
\newtheorem{proposition}[theorem]{Proposition}
\newtheorem{conjecture}[theorem]{Conjecture}

\newtheorem{lemma}[theorem]{Lemma}

\newcommand{\Pl}{\mathbb P}
\newcommand{\E}{ \mathbb E}

\def\E{\mathbb{E}}
\def\Var{\mathrm{Var}}
\def\Cov{\mathrm{Cov}}
\usepackage{chngcntr}
\counterwithout{equation}{section} % undo numbering system provided by phstyle.cls
\usetikzlibrary{arrows, automata,positioning,calc,shapes,decorations.pathreplacing}
  \colorlet{greencolor}{green!50!black}
  \colorlet{textcolor}{red}
  \colorlet{tancolor}{orange!80!black}
  \colorlet{bluecolor}{blue}
\usepackage{amsmath,amsfonts,amssymb,amsthm,epsfig,epstopdf,url,array}

\definecolor{plotcolor1}{rgb}{0,0.447,0.741}
\definecolor{plotcolor2}{rgb}{0.741,0,0.447}
\definecolor{plotcolor3}{rgb}{0,0.741,0.294}
\definecolor{plotcolor4}{rgb}{0.741,0.294,0}
\definecolor{plotcoloraux}{rgb}{0.447,0.447,0.447}
\tikzset{plotstyle1/.style={color=plotcolor1,solid,line width=1.0pt}}
\tikzset{plotstyle2/.style={color=plotcolor2,densely dashed,line width=1.0pt}}
\tikzset{plotstyle3/.style={color=plotcolor3,dotted,line width=1.0pt}}
\tikzset{plotstyle4/.style={color=plotcolor4,loosely dashed,line width=1.0pt}}
\tikzset{auxlines/.style={color=plotcoloraux,solid,line width=0.5pt}}
\newcommand{\markersize}{0.7pt}
\tikzset{discretemarkers/.style={mark=*,mark options={solid},mark size=\markersize}}
\usepackage{makecell} %for changing the lines in table
\title{Stationary Markovian Arrival Processes,\\ Results and Open Problems}

\author{Azam Asanjarani\footnote{The University of Auckland. Email: azam.asanjarani@auckland.ac.nz}, 
Yoni Nazarathy\footnote{The University of Queensland. Email: y.nazarathy@uq.edu.au}.}

\date{ April 15, 2019}

\begin{document}
\maketitle

\begin{abstract}
We consider two classes of irreducible Markovian arrival processes specified by the matrices $C$ and $D$. The Markov Modulated Poison Process (MMPP) and the Markovian Switched Poison Process (MSPP). The former exhibits a diagonal $D$ while the latter exhibits a diagonal $C$. 
For these two classes, we consider the following statements: (I) Overdispersion of the counts process. (II) A non-increasing hazard rate of the stationary inter-event time. (III) The squared coefficient of variation of the event stationary process is greater or equal to unity. (IV)  A stochastic order showing that the time stationary inter-arrival time dominates the event-stationary time.
For general MSPPs and two-state MMPPs, we show that (I)-(IV) hold. Then for general MMPPs, it is easy to establish (I), while (II) is false due to a counter-example of Miklos Telek and Illes Horvath. For general simple point processes, (III) follows from (IV). For MMPPs we conjecture and numerically test that (IV) and thus (III) hold. 
Importantly, modeling folklore has often treated MMPPs as ``bursty'' and implicitly assumed that (III) holds. However, this is still an open question.

\end{abstract}


%%%%%%%%%%%%%%%%%%%%%%%%%%%%%%%%%%%%%%%%%%%%%%%%
%%%%%%%%%%%%%%%%%%%%%%%%%%%%%%%%%%%%%%%%%%%%%%%%
%%%%%%%%%%%%%%%%%%%%%%%%%%%%%%%%%%%%%%%%%%%%%%%%
%%%%%%%%%%%%%%%%%%%%%%%%%%%%%%%%%%%%%%%%%%%%%%%%
%%%%%%%%%%%%%%%%%%%%%%%%%%%%%%%%%%%%%%%%%%%%%%%%
%%%%%%%%%%%%%%%%%%%%%%%%%%%%%%%%%%%%%%%%%%%%%%%%
%%%%%%%%%%%%%%%%%%%%%%%%%%%%%%%%%%%%%%%%%%%%%%%%
%%%%%%%%%%%%%%%%%%%%%%%%%%%%%%%%%%%%%%%%%%%%%%%%
%%%%%%%%%%%%%%%%%%%%%%%%%%%%%%%%%%%%%%%%%%%%%%%%
\section{Introduction}
\label{sec:intro}

Point processes on the line, generated by transitions of Continuous Time Markov Chains (CTMCs) have been studied intensely by the applied probability community over the past few decades under the umbrella of Matrix Analytic Methods (MAM), see e.g. \cite{latouche1999introduction}. These have been applied to teletraffic \cite{akar1998matrix}, business networks \cite{herbertsson2007pricing}, social operations research \cite{xing2013operations}, and biological systems \cite{olsson2015equilibrium}. The typical model referred to as the Markovian Arrival Process (MAP) is comprised of a finite state irreducible CTMC which generates events at selected instances of state change and/or according to Poisson processes modulated by the CTMC. MAPs have been shown to be dense in the class of point processes so that they can essentially approximate any point process, \cite{asmussen1993marked}. Yet at the same time, they are analytically tractable and may often be incorporated effectively within more complex stochastic models \cite{neuts1979versatile}. 
%Some notable descriptions of MAPs are in \cite{asmussen2003applied} (Chapter XI), \cite{he2014fundamentals} (Chapter 2), and  \cite{latouche1999introduction} (Chapter 3).
%, \cite{lucantoni1991new}, and \cite{neuts1979versatile}.

In general, treating point processes as {\em stationary} often yields a useful mathematical perspective which matches scenarios when there is no known dependence on time. In describing a point process we use  $N(t)$ to denote the number of events during $[0,t]$ and further use the sequence $\{T_n\}$ to denote the sequence of inter-event times. Two notions of stationarity are useful in this respect. Roughly, a point process is {\em time-stationary} if the distribution of the number of events within a given interval does not depend on the location of the interval; that is if $N(t_1+s)-N(t_1)$ is distributed as $N(t_2+s)-N(t_2)$ for any non-negative $t_1, t_2$ and $s$. A point process is {\em event-stationary} if the joint distribution of $T_{k_1},\ldots,T_{k_n}$ is the same as that of $T_{k_1+\ell},\ldots,T_{k_n+\ell}$ for any integer sequence of indices $k_1, \ldots,k_n$ and any integer shift $\ell$. For a given model of a point process, one may often consider either the event-stationary or the time-stationary case. The probability laws of both cases agree in the case of the Poisson process. However, this is not true in general. 
For MAPs, time-stationarity and event-stationarity are easily characterized by the initial distribution of the background CTMC. Starting it at its stationary distribution yields time-stationarity and starting at the stationary distribution of the embedded Markov chain (jump chain) yields event-stationarity.

A common way to parameterize MAPs is by considering the generator, $Q$, of an irreducible finite state CTMC and setting $Q= C + D$. Roughly speaking, the matrix $C$ determines state transitions without event counts and the matrix $D$ determines event counts. Such parameterization hints at considering two special cases: Markov Modulated Poisson Processes (MMPP) arising from a diagonal matrix $D$, and Markovian Switched Poisson Processes (MSPP) arising from a diagonal matrix $C$. 

MMPPs are a widely used class of processes in modelling and are a typical example of a Cox process, also known as  a doubly stochastic Poisson process,  \cite{grandell2006doubly} and \cite{tang2009markov}. For a detailed outline of a variety of classic MMPP results, see~\cite{fischer1993markov} and references therein. MSPPs were introduced in \cite{dan1991counter} and to date, have not been as popular for modeling. However, the duality of diagonal $D$ vs. diagonal $C$ motivates us to consider and contrast both these processes. We also note that hyper-exponential renewal processes are special cases of MSPPs as well as Markovian Transition Counting Processes (as introduced in \cite{asanjarani2016queueing}).

Our focus in this paper is on second order properties of MMPPs and MSPPs and related traits. Consider the squared coefficient of variation and the limiting index of dispersion of counts given by,
\begin{equation}
\label{eq:3535}
c^2 = \frac{\Var(T_1^{\boldsymbol{\alpha}})}{\E^2\,[ T_1^{\boldsymbol{\alpha}}]},
\qquad
\mbox{and}
\qquad
d^2 = \lim_{t \to \infty} \frac{\Var(N(t))}{\E[N(t)]},
\end{equation} 
where $T_1^{\boldsymbol{\alpha}}$ is the time of the first event, taken from the event stationary version. Modelling folklore of MMPP sometimes assumes that $c^2 \ge 1$. This is perhaps due to the fact that $d^2 \ge 1$ is straightforward to verify and the similarity between these measures (for example for a renewal process, $c^2 = d^2$). However, as we highlight in this paper, establishing such ``burstiness'' properties is not straightforward.

A related property is having $T_1^\alpha$ exhibit Decreasing Hazard Rate (DHR), where for a random variable with PDF $f(t)$ and CDF $F(t)$ the hazard rate is,
\[
h(t)=\frac{f(t)}{1-F(t)}.
\]
A further related property is the stochastic order, $T_1^\pi \ge_{\mbox{st}} T_1^\alpha$ where $T_1^\pi$ is the first event time in the time-stationary version.  We denote the properties as follows:
\begin{description}
\item (I) $d^2 \ge 1$.
\item (II) $T_1^{\boldsymbol{\alpha}}$ exhibits DHR. 
\item (III) $c^2 \ge 1$.
\item (IV) The stochastic order $T_1^{\boldsymbol{\pi}} \ge_{\mbox{st}} T_1^{\boldsymbol{\alpha}}$.
\end{description} 

All these properties are related and in this paper we highlight relationships between (I), (II), (III) and (IV) and establish the following: For MSPPs and MMPPs of order $2$ we show that (I)--(IV) holds. For general MMPPs it is known that (I) holds however, a counter-example of Miklos Telek and  Illes  Horvath shows that (II) does not hold and we conjecture (and numerically test) that (III) and (IV) holds.

Our interest in this class of problems stemmed from relationships between different types of MAPs as in \cite{nazarathy2008asymptotic} and \cite{asanjarani2016queueing}. Once it became evident that $c^2 \ge 1$ for MMPPs is an open problem even though it is acknowledged as a modeling fact in folklore, we searched for alternative proof avenues. This led to the stochastic order in (IV) as well as to considering DHR properties (the latter via communication with Miklos Telek and Illes Horvaths).

The remainder of the paper is structured as follows. In Section~\ref{sec2} we present preliminaries, focusing on the relationships between properties (I) -- (IV) as well as defining MMPPs and MSPPs. In Section~\ref{sec3} we present our main results and the conjecture. We close in Section~\ref{sec4}.

%QQQQ NOTE: **** maybe bring back if "reader" is interested in why SCV is important ****
%Within the queueing network approximation literature (see for example \cite{whitt1983queueing} and more recent papers citing that work), the squared coefficient of variation plays a central role. It has even been adopted within the stochastic manufacturing modelling community as the standard measure of variability of flows (see \cite{buzacott1993stochastic}). For queueing network approximations, the typical idea is to approximate streams of customers moving between nodes in a queueing network by renewal processes or other similar processes. In doing so, such heuristic schemes often focus heavily on the squared coefficient of variation.  For example, in a recent paper \cite{kim2011modeling}, the author even suggests (as an approximation) to use MMPPs and states (see page 2 of that paper) ``{\em For example, the squared coefficient of variation (SCV) of stationary intervals must be greater than or equal to $1$.}'', referring to MMPPs. Hence we believe that in stochastic modelling and engineering folklore, $c^2$ is a typical important quantity.

%%%%%%%%%%%%%%%%%%%%%%%%%%%%%%%%%%%%%%%%%%%%%%%%
%%%%%%%%%%%%%%%%%%%%%%%%%%%%%%%%%%%%%%%%%%%%%%%%
%%%%%%%%%%%%%%%%%%%%%%%%%%%%%%%%%%%%%%%%%%%%%%%%
%%%%%%%%%%%%%%%%%%%%%%%%%%%%%%%%%%%%%%%%%%%%%%%%
%%%%%%%%%%%%%%%%%%%%%%%%%%%%%%%%%%%%%%%%%%%%%%%%
%%%%%%%%%%%%%%%%%%%%%%%%%%%%%%%%%%%%%%%%%%%%%%%%
%%%%%%%%%%%%%%%%%%%%%%%%%%%%%%%%%%%%%%%%%%%%%%%%
%%%%%%%%%%%%%%%%%%%%%%%%%%%%%%%%%%%%%%%%%%%%%%%%
%%%%%%%%%%%%%%%%%%%%%%%%%%%%%%%%%%%%%%%%%%%%%%%%
\section{Preliminaries}
\label{sec2}

Consider first properties (I)--(IV) and their relationships. With an aim of establishing property (III), $c^2 \ge 1$, there are several possible avenues based on properties (I), (II) and (IV). We now explain these relationships.

\vspace{5pt}
\noindent
\paragraph*{Using property (I):} First, from the theory of simple point processes on the line, note the relationship between $d^2$ and $c^2$: 
\begin{equation}\label{eq:dcR}
d^2 = c^2\Big(1+2\sum_{j=1}^\infty \frac{\Cov(T_0^{\boldsymbol{\alpha}},T_j^{\boldsymbol{\alpha}})}{\Var(T_0^{\boldsymbol{\alpha}})}\Big).
\end{equation}
However, the autocorrelation structure is typically intractable and hence does not yield results. If we were focusing on a renewal process where $T_i$ and $T_j$ are independent for $i \neq j$ then this immediately shows that $d^2 = c^2$. Our focus is broader and hence property (I) indicating that $d^2 \ge 1$ does not appear to be of use.

\vspace{5pt}
\noindent
\paragraph*{Using property (II):} An alternative way is to consider property (II) and use the fact that for any DHR random variable we have $c^2 \ge 1$ (see \cite{stoyan1983comparison}). Hence if property (II) holds then (III) holds.

%QQQQ1 - look at reference and discuss (maybe explicit theorem number from the reference is needed) or similar.

\vspace{5pt}
\noindent
\paragraph*{Using property (IV):} We have the following Lemma, implying that (III) is a consequence of the stochastic order (IV).

\begin{lemma}
\label{lem:soscv}
Consider a simple non-transient point process on the line, and let $T_1^{\boldsymbol{\pi}}$, $T_1^{\boldsymbol{\alpha}}$ represent the first inter-event time in the time-stationary case and event-stationary case respectively. Then $c^2 \ge 1$ if and only if $\E[T_1^{\boldsymbol{\pi}}] \geq \E[T_1^{\boldsymbol{\alpha}}]$.
\end{lemma}
\begin{proof}
From point process theory (see for example,  Eq. (3.4.17) of \cite{daley2007introduction}), it holds $$
\E[T_1^{\boldsymbol{\pi}}]=\frac{1}{2}\lambda^* \E\big[\big(T_1^{\boldsymbol{\alpha}}\big)^2\big],
$$
where,
\[
\lambda^* = \lim_{t \to \infty} \frac{E\big[N[0,t]\big]}{t} = \frac{1}{\E[T_1^{\boldsymbol{\alpha}}]}.
\]
Now,
\[
c^2 =\frac{\E[\big(T_1^{\boldsymbol{\alpha}}\big)^2]-\big(\E[T_1^{\boldsymbol{\alpha}}]\big)^2}{\big(\E[T_1^{\boldsymbol{\alpha}}]\big)^2} = 2 \frac{\E[T_1^{\boldsymbol{\pi}}]}{\E[T_1^{\boldsymbol{\alpha}}]} - 1,
\]
and we obtain the result.
\end{proof}

\paragraph*{MAPs:} We now describe Markovian Arrival Process (MAPs). 
%A MAP is a mathematical model based on a Markov chain, used for modelling events occurring over time. 
A MAP of order $p$ (MAP$_p$) is generated by a two-dimensional Markov process  $\{(N(t), X(t)); t \geq 0\}$ on state space $\{0,1, 2, \cdots\}\times \{1,2, \cdots, p\}$. The counting process $N(\cdot)$ counts  the number of  ``events''  in $[0,t]$ with  $\Pl(N(0)=0)=1$. The phase process $X(\cdot)$ is an irreducible CTMC with state space $\{1, \ldots, p\}$, initial distribution $\boldsymbol{\eta}$  and  generator matrix $Q$. A MAP is characterized by parameters $(\boldsymbol{\eta}, C,D)$, where the matrix  $C$ has negative diagonal elements and non-negative off-diagonal elements and records the rates of phase transitions which are not associated with an event. The matrix $D$ has non-negative elements and describes the changes of the phase process with an event (increase of $N(t)$ by 1). Moreover, we have $Q=C+D$. More details are in \cite{asmussen2003applied} (Chapter~XI) and \cite{he2014fundamentals} (Chapter~2).
 
MAPs are attractive due to the tractability of many of their properties, including distribution functions, generating functions, and moments of both $N(\cdot)$ and the sequence of inter-event times $\{T_n\}$. 
%The density of the time until the first event is 
%$
%f(t) = \boldsymbol{\eta} e^{C t} D {\mathbf 1}
%$.
%This is in the form of a Phase-type (PH) distribution  $PH(\boldsymbol{\eta}, C)$. Here $\boldsymbol{\eta}$ is the so-called initial distribution vector and the sub-generator $C$ defines transition rates of a CTMC among its transient states (see Chapter 1 of \cite{he2014fundamentals} for more details on PH distributions).  Note that here since $Q {\mathbf 1} = {\mathbf 0}$, the exit vector $-C {\mathbf 1}$ can be represented by $D {\mathbf 1}$ as well. 
%
Since $Q$ is assumed irreducible and finite, it has a unique stationary distribution $\boldsymbol{\pi}$ satisfying $\boldsymbol{\pi} Q = \mathbf{0}'$, $\boldsymbol{\pi} \mathbf{1} = 1$.  Note that from $Q {\mathbf 1} = \mathbf {0}'$ we have  $-C {\mathbf 1}=D {\mathbf 1}$.
 Of further interest is the embedded discrete-time Markov chain with irreducible stochastic matrix $P = (-C)^{-1}D$ and stationary distribution $\boldsymbol{\alpha}$, where $\boldsymbol{\alpha} P=\boldsymbol{\alpha}$ and $\boldsymbol{\alpha} \mathbf{1}=1$. 

Observe the relation between the stationary distributions $\boldsymbol{\pi}$ and $\boldsymbol{\alpha}$:
\begin{equation}\label{Eq:pi-alpha}
\boldsymbol{\alpha}=\frac{\boldsymbol{\pi} D}{\boldsymbol{\pi} D \mathbf{1}} \qquad \text{and}  \qquad \boldsymbol{\pi}=\frac{\boldsymbol{\alpha} (-C)^{-1}}{\boldsymbol{\alpha} (-C)^{-1}\mathbf{1}}=\lambda^*\boldsymbol{\alpha} (-C)^{-1},
\end{equation}
where $\lambda^* =  \boldsymbol{\pi} D {\mathbf 1} = - \boldsymbol{\pi} C {\mathbf 1}$.

The following known proposition, as distilled from the literature (see for example \cite{asmussen2003applied}, Chapter XI) provides the key results of MAPs that we use in this paper. It shows that $T_1$ is a Phase Type (PH) random variable with parameters $\boldsymbol{\eta}$ for the initial distribution of the phase and $C$ for the sub-generator matrix. It further shows that the initial distribution of the phase process may render the MAP as time stationary or event stationary.

\begin{proposition}
Consider a MAP with parameters ($\boldsymbol{\eta}$,$C$,$D$), then
\begin{equation}
\label{eq:T1dist}
\Pl(T_1 > t) = \boldsymbol{\eta} e^{C t} {\mathbf 1}.
\end{equation}
Further, if $\boldsymbol{\eta} = \boldsymbol{\pi}$ then the MAP is time-stationary and if $\boldsymbol{\eta}=\boldsymbol{\alpha}$ it is event stationary, where $\boldsymbol{\pi}$ and $\boldsymbol{\alpha}$ are associated stationary distributions.
\end{proposition}

Note that for such a $PH(\boldsymbol{\eta}, C)$ random variable the density $f(t)$ and the hazard rate $h(t)$, are respectively,
\[
f(t) = \boldsymbol{\eta} e^{C t} D {\mathbf 1},
\qquad
h(t) = \frac{\boldsymbol{\eta}e^{Ct} D \mathbf{1}}{\boldsymbol{\eta}e^{Ct}  \mathbf{1}}.
\]
Further, as may be used for showing DHR, the derivative of the hazard rate is,
\begin{equation}
\label{eq:derHazard}
h^{\prime}(t)= \frac{\boldsymbol{\eta} C e^{Ct} (-C) \mathbf{1}  \,\,
\boldsymbol{\eta}e^{Ct}  \mathbf{1} - \boldsymbol{\eta} C e^{Ct}  \mathbf{1}\,\, \boldsymbol{\eta}e^{Ct} (-C) \mathbf{1}
}{(\boldsymbol{\eta}e^{Ct}  \mathbf{1})^2}.
\end{equation}

We now describe second-order properties associated with each case. 
\paragraph*{Event-Stationary Case:}     
    The MAP is event-stationary\footnote{Sometimes an event-stationary MAP is referred to as an interval-stationary MAP, see for instance \cite{fischer1993markov}.} if 
    %there is an event at time $t=0$ and 
$\boldsymbol{\eta}=\boldsymbol{\alpha}$. In this case, the (generic) inter-event time is phase-type distributed, $PH(\boldsymbol{\alpha}, C)$ and thus has $k$-th moment:
$$
M_k=\E[T_n^k]=k! \boldsymbol{\alpha} (-C)^{-k}\mathbf{1}
%=k! \, \frac{1}{\lambda^*}( -\boldsymbol{\pi} C) (-C)^{-k} {\mathbf 1}
=(-1)^{k+1} \, k! \, \frac{1}{\lambda^*} \boldsymbol{\pi} \big(C^{-1}\big)^{k-1} {\mathbf 1},
$$
with the first and second moments (here represented in terms of $\boldsymbol{\pi}$ and $C$):
\[
M_1=\frac{1}{\lambda^*} \boldsymbol{\pi}  {\mathbf 1} = \frac{1}{\lambda^*},
\qquad
M_2=2 \frac{1}{\lambda^*} \boldsymbol{\pi} (-C)^{-1} {\mathbf 1}.
\]
The squared coefficient of variation (SCV) of events (intervals)  has a simple formula: 
\begin{equation}\label{Eq:SCV-Interval}
c^2+1 = \frac{M_2}{M_1^2} =  
\frac{-2 \, (1/\lambda^*) \, \boldsymbol{\pi} \, C^{-1} {\mathbf 1}}{(1/\lambda^*)^2}
= 2 \boldsymbol{\pi} C {\mathbf 1} \boldsymbol{\pi} C^{-1} {\mathbf 1}.
\end{equation}

\paragraph*{Time-Stationary Case:} 
 A MAP with parameters $(\boldsymbol{\eta}, C,D)$ is time-stationary \index{time-stationary MAP} if $\boldsymbol{\eta}=\boldsymbol{\pi}$. In the time-stationary  case ($\boldsymbol{\eta}=\boldsymbol{\pi}$), we have  (see \cite{asmussen2003applied}):
\begin{align}
\label{Eq:Mean}
\E[N(t)]&= {\boldsymbol{\pi}} D \mathbf{1}\,t,\\
\label{Eq:Var}
\Var\big(N(t)\big)&=\{{\boldsymbol{\pi}}D \mathbf{1}+2\, {\boldsymbol{\pi}}D D_Q^{\sharp} D \mathbf{1}\}\,t- 2 {\boldsymbol{\pi}}D D_Q^{\sharp} D_Q^{\sharp}(t) D \mathbf{1},
\end{align}
 where  $D_Q^{\sharp}$ is the \textit{deviation matrix}\index{deviation matrix} associated with  $Q$ defined by the following formula.
\begin{equation} 
\label{Eq:deviation}
D_Q^{\sharp}=\lim_{t\rightarrow\infty} D^{\sharp}_Q(t)=\int_0^{\infty}(e^{Qu}-\mathbf{1}{\boldsymbol{\pi}})\, du.
\end{equation}
 Note that in some sources, for instance  \cite{asmussen2003applied} and \cite{narayana1992first},  the variance formula \eqref{Eq:Var}  is presented in terms of the  matrix $Q^{-}:=(\mathbf{1}{\boldsymbol{\pi}} - Q)^{-1}$.  The relation between these two matrices is  $Q^{-}=D_Q^{\sharp} +\mathbf{1}{\boldsymbol{\pi}}$, see \cite{coolen2002deviation}.
  
Applying \eqref{Eq:Mean} and \eqref{Eq:Var}, we can write $d^2$ in terms of a MAP parameters as:
\begin{equation}\label{eq:d}
d^2= 1+
\frac{2}{\lambda^*}\, {\boldsymbol{\pi}}D D_Q^{\sharp} D \mathbf{1}.
\end{equation}
%

%%%%%%%%%%%%%%%%%%%%%%%%%%%%%%%%%%%%%
%%%%%%%%%%%%%%%%%%%%%%%%%%%%%%%%%%%%%
%%%%%%%%%%%%%%%%%%%%%%%%%%%%%%%%%%%%%
%%%%%%%%%%%%%%%%%%%%%%%%%%%%%%%%%%%%%
\paragraph*{MMPP:} A MAP with a diagonal matrix $D$ is an MMPP.
%; namely only ``self-transitions'' generate events. 
MMPPs correspond to doubly-stochastic Poisson processes (also known as Cox processes) where the modulating process is driven by a CTMC. 
MMPPs have been used extensively in stochastic modelling and analysis, see for example \cite{fischer1993markov}.
The parameters of an MMPP$_p$ are $D= \text{diag}(\lambda_i)$, where $\lambda_i \geq 0$ for $i=1, \ldots, p$,  and  $C=Q-D$. Here, $Q$ is the generator matrix of a CTMC. 
For MMPPs, \eqref{Eq:Mean} and \eqref{Eq:Var} can be simplified by using the following relations:
$$
\boldsymbol{\pi} D \mathbf{1}=\sum_{i=1}^p {\pi}_i \lambda_i,
\qquad D \mathbf{1}=\boldsymbol{\lambda}=(\lambda_1, \cdots, \lambda_p)^\prime,
\qquad \boldsymbol{\pi} D=({\pi}_1 \lambda_1, \cdots, {\pi}_p\lambda_p)\,.
$$

\paragraph*{MSPP:}  A MAP with a diagonal  matrix $C$ is is an MSPP.  For MSPP$_p$ events switch between $p$ Poisson processes with  rates $\lambda_1, \cdots, \lambda_p$, where each switch also incurs an event. Here as in MMPPs we denote the diagonal elements of $D$ via $\lambda_1, \cdots, \lambda_p$. However, unlike MMPPs, (irreducible) MSPPs don't have a diagonal $D$.  We also remark that the modulation in the MSPP is of a discrete nature and it occurs at certain event epochs of the counting process, whereas the modulation of the MMPP is performed at epochs without events. See \cite{artalejo2010markovian} and \cite{he2014fundamentals}.

As our research attempts have shown, analyzing MSPPs is considerably easier than MMPPs, because a diagonal $C$ is much easier to handle than a non-diagonal $C$ and in an (irreducible) MMPP, $C$ must be non-diagonal.

\paragraph*{Properties (I)-(IV) for MAPs:} Using the results above, for any irreducible MAP with matrices $C$ and $D$ we have that the main properties (I)-(IV) of this paper can be formulated as follows:
\begin{align}
 \label{Eq:d2}(I)&\qquad  {\boldsymbol{\pi}}D D_Q^{\sharp} D \mathbf{1} \ge 0, \\
 \label{Eq:HDR}
 (II)&\qquad \boldsymbol{\alpha} C e^{Ct} (-C) \mathbf{1}  \,\,
\boldsymbol{\alpha}e^{Ct}  \mathbf{1} + (\boldsymbol{\alpha} C e^{Ct}  \mathbf{1})^2 \leq 0 \qquad \forall t \ge 0, \\
   \label{Eq:c2} (III)&\qquad \boldsymbol{\pi} C {\mathbf 1} \boldsymbol{\pi} C^{-1} {\mathbf 1} \ge 1, \\
 \label{Eq:SO}(IV)&\qquad \boldsymbol{\pi} e^{C t} \mathbf{1} \ge \boldsymbol{\alpha} e^{C t} \mathbf{1},
\qquad \forall t \ge 0. 
\end{align}

%%%%%%%%%%%%%%%%%%%%%%%%%%%%%%%%%%%%%%%%%%%%%%%%
%%%%%%%%%%%%%%%%%%%%%%%%%%%%%%%%%%%%%%%%%%%%%%%%
%%%%%%%%%%%%%%%%%%%%%%%%%%%%%%%%%%%%%%%%%%%%%%%%
%%%%%%%%%%%%%%%%%%%%%%%%%%%%%%%%%%%%%%%%%%%%%%%%
%%%%%%%%%%%%%%%%%%%%%%%%%%%%%%%%%%%%%%%%%%%%%%%%
%%%%%%%%%%%%%%%%%%%%%%%%%%%%%%%%%%%%%%%%%%%%%%%%
%%%%%%%%%%%%%%%%%%%%%%%%%%%%%%%%%%%%%%%%%%%%%%%%
%%%%%%%%%%%%%%%%%%%%%%%%%%%%%%%%%%%%%%%%%%%%%%%%
%%%%%%%%%%%%%%%%%%%%%%%%%%%%%%%%%%%%%%%%%%%%%%%%
\section{Main Results}
\label{sec3}

We now present results for MSPP and MMPP$_2$ for properties (I)-(IV) as presented in the introduction. Establishing property (I), $d^2 \ge 1$ is not a difficult task for both MMPPs and MSPPs:

\begin{proposition}
\label{eq:pp22}
MMPP and MSPP processes have $d^2 \ge 1$.
\end{proposition}

\begin{proof}
This is a well-known result that for all doubly stochastic Poisson processes (Cox processes), $d^2 \ge 1$. So, we have the proof for an MMPP,  for instance  see Chapter 6 of \cite{kingman1993poisson}.

For an MSPP, using the fact that for a given MMPP, we have $d^2\geq 1$, results in:
\begin{equation}\label{eq:Ddiag}
{\boldsymbol{\pi}}D D_Q^{\sharp} D \mathbf{1} \geq 0,\quad \text{for any diagonal non-negative matrix $D$.}
\end{equation}
On the other hand,  all MAPs satisfy $
{\boldsymbol{\pi}}D D_Q^{\sharp} D \mathbf{1}={\boldsymbol{\pi}}(-C) D_Q^{\sharp} (-C) \mathbf{1}$. Since for an MSPP, $-C$  is a diagonal non-negative matrix, from \eqref{eq:Ddiag} we have  \eqref{Eq:d2}.
\end{proof}

It isn't difficult to show that property (II), DHR holds for MSPP:

\begin{proposition}
\label{eq:pp22}
For an MSPP the hazard rate of the stationary inter-event time is non-increasing.
\end{proposition}
\begin{proof}
%
%$$
%h^{\prime}(t)= \frac{\boldsymbol{\alpha} C e^{Ct} (-C) \mathbf{1}  \,\,
%\boldsymbol{\alpha}e^{Ct}  \mathbf{1} - \boldsymbol{\alpha} C e^{Ct}  \mathbf{1}\,\, \boldsymbol{\alpha}e^{Ct} (-C) \mathbf{1}
%}{(\boldsymbol{\alpha}e^{Ct}  \mathbf{1})^2}.
%$$
%  The sign of $h^{\prime}(t)$ just depends on the nominator. We can write the nominator as:
%\begin{equation}\label{eq:dif1}
%-\boldsymbol{\alpha} C^2 e^{Ct}  \mathbf{1}\,\, \boldsymbol{\alpha}  e^{Ct}  \mathbf{1} + ( \boldsymbol{\alpha} C e^{Ct}  \mathbf{1}  )^2.
%\end{equation}
Denote the diagonal matrix $C$ with $C=diag(-c_i)$ and the positive elements of the column vector $e^{Ct}  \mathbf{1}$ with $\mathbf{u}$.
So, Eq.~\eqref{Eq:HDR} can be written element-wise as:
\begin{equation}\label{eq:dif2}
-(\sum _{i=1}^p\alpha_i c_i^2 u_i)\,(\sum_{i=1}^p \alpha_i u_i)+(\sum_{i=1}^p \alpha_i c_i u_i)^2.
\end{equation}
Denoting $v_i=\alpha_i u_i$ and assuming $p_i=\frac{v_i}{\sum_{i=1}^p v_i}$ results in:
$$
-(\sum _{i=1}^p c_i^2 p_i)+(\sum _{i=1}^p c_i p_i)^2.
$$
The above expression can be viewed as the minus variance of a random variable that takes values $c_i$ with probability $p_i$. Therefore, we have \eqref{Eq:HDR}.

\end{proof}

However, somewhat surprisingly, MMPPs don't necessarily possess DHR. An exception is MMPP$_2$ as shown in Proposition~\ref{prop:mmpp2}. However for higher order MMPPs DHR doesn't always hold. The gist of the following example was communicated to us by Milkos Telek and Illes Horvath. Set 
\begin{equation}
\label{Example}
Q =\left(
\begin{array}{cccc}
-1 & 1 & 0 & 0 \\ 
0 & -1 & 1 & 0\\
0 & 0 & -1 & 1\\
1& 0&0 & -1 
  \end{array}
\right)
\qquad
\mbox{and}
\qquad
D =\left(
\begin{array}{cccc}\displaystyle
0.01 & 0 &\, \,0 & \,\,\,\,0 \\ 
0 & 0.01 &\, \,0 &\, \,\,\,0\\
0 & 0 & \,\,1 & \,\,\,\,0\\
0& 0&\,\,0 & \,\,\,\,1
  \end{array}
\right).
%D = \mbox{diag}\,(0.01,0.01,1,1)
\end{equation}
As shown in Figure~\ref{Fig:example}, the hazard rate function for an MMPP with the above matrices is not monotone. Hence at least for general MMPPs, trying to show (III), $c^2\ge 1$, via hazard rates is not a viable avenue.
%%%%%%%%%%%%%%%%%
\begin{figure}
\center
\includegraphics[scale=0.4]{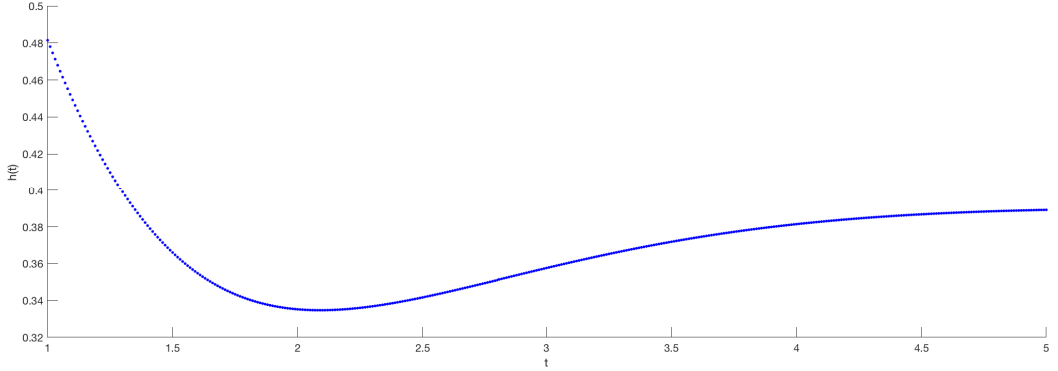}
\caption{{{\small The hazard rate of the MMPP in \eqref{Example} is not monotone.}}}
\label{Fig:example}
\end{figure}

%If you include EPS (encapsulated postscript) figures in your paper,
%then please use the following commands:
%\begin{figure}
%\begin{center}
%\includegraphics{.eps}
%\caption{Caption text.}\label{}
%\end{center}
%\end{figure}

Since hazards rates don't appear to be a viable paths for establishing (III) for MMPPs, an alternative may be to consider the stochastic order (IV). Starting with MSPPs, we see that this property holds.

\begin{proposition}
\label{eq:msppSO}
For an MSPP $T_1^{\boldsymbol{\pi}} \ge_{\mbox{st}} T_1^{\boldsymbol{\alpha}}$.
\end{proposition}

\begin{proof}

Using \eqref{Eq:SO}, the claim is,
\begin{equation}
\label{eq:351}
(\boldsymbol{\pi} - \boldsymbol{\alpha})e^{C t} \mathbf{1} \ge  0,
\qquad \forall t \ge 0.
\end{equation}
Without loss of generality we assume  that there is an order  $0<c_1\leq c_2 \leq \cdots \leq c_p$ (with $c_i \neq c_j$ for some $i,j$) for diagonal elements of matrix $(-C)$.
There is a possibility  that for $1 < p^{\prime} < p$, $0=c_1=c_2 = \ldots =c_{p^{\prime}-1}$, and $0<c_{p^{\prime}}$, however,
in the rest of the proof, we assume that $p^{\prime}=1$,  meaning that all $c_i$ are strictly positive. Adapting to the case of $p^{\prime}>1$ is straightforward. 

Now, $\{\lambda^*-c_i\}_{i=1, \cdots, p}$ is a non-increasing sequence and therefore in the sequence $\{\pi_i-\alpha_i\}=\{\frac{\pi_i}{\lambda^*}(\lambda^*-c_i)\}$ when an element $\pi_k-\alpha_k$ is negative, all the elements $\pi_i-\alpha_i$ for $i\geq k$ are negative.  
Moreover, both $\boldsymbol{\pi}$ and $\boldsymbol{\alpha}$ are probability vectors, so  $(\boldsymbol{\pi}-\boldsymbol{\alpha})\mathbf{1}=\sum_i(\pi_i-\alpha_i)=0$. Therefore,
 at least the first element in the sequence $\{\pi_i-\alpha_i\}=\{\frac{\pi_i}{\lambda^*}(\lambda^*-c_i)\}$ is positive. 
Hence, there exists an index $1 < k \leq  p$  such that $\pi_i-\alpha_i$ for $i=1, \cdots, k-1$ is non-negative and for $i=k,\cdots, p$ is negative. Therefore,  we have:

\begin{align*}
  ({\pi} - {\boldsymbol\alpha}) e^{Ct} \mathbf{1}
 &=\underbrace{\sum_{i=1}^{k-1}({\pi_i}-{ \alpha_i})e^{-c_i t}}_{\text{non-negative}}+\underbrace{\sum_{i=k}^p({\pi_i}-{ \alpha_i})e^{-c_i t}}_{\text{negative}}\\
& =\underbrace{\sum_{i=1}^{k-1}({\pi_i}-{ \alpha_i})e^{-c_i t}}_{\text{non-negative}}-\underbrace{\sum_{i=k}^p({\alpha_i}-{\pi_i})e^{-c_i t}}_{\text{non-negative}}.
 \end{align*}
 Assume: $({\pi} - {\alpha}) e^{Ct} \mathbf{1} <0$ :
 \begin{equation}
 \label{Eq:neg}
 \sum_{i=1}^{k-1}({\pi_i}-{\alpha_i})e^{-c_i t}
< 
\sum_{i=k}^p({\alpha_i}-{\pi_i})e^{-c_i t}.
 \end{equation}
 Then, since $0<c_1\leq c_2 \leq \cdots \leq c_p\,$,  we have 
 $e^{-c_1 t} \geq e^{-c_2 t} \geq \cdots  \geq e^{-c_p t} $. Now from \eqref{Eq:neg} we can conclude that:
 $$
 \sum_{i=1}^{k-1}({\pi_i}-{ \alpha_i})e^{-c_{k-1} t} < \sum_{i=k}^p({\alpha_i}-{\pi_i})e^{-c_k t}.
 $$
Using the fact that $\sum_{i=k}^p({\alpha_i}-{\pi_i}) =\sum_{i=1}^{k-1}({\pi_i}-{ \alpha_i})$, results in:
$$
e^{-c_{k-1}t} \,\sum_{i=1}^{k-1}({\pi_i}-{ \alpha_i})
< 
e^{-c_k t}\, \sum_{i=1}^{k-1}({\pi_i}-{ \alpha_i}),
 $$
 which is not true.
Consequently, the assumption $({\pi} - {\alpha}) e^{Ct} \mathbf{1} <0$ is not true and hence  \eqref{Eq:SO} holds.

\end{proof}

Hence via Lemma~\ref{lem:soscv} or alternatively via the DHR property in Proposition~\ref{eq:pp22} we have:

\begin{corollary}
\label{eq:msppSO}
For an MSPP $c^2 \ge 1$.
\end{corollary}

In fact, for MSPPs this is an easy result and it can also be proved independently by using the Cauchy-Schwarz inequality. Further, we can find an upper bound:
\begin{proposition}
\label{prop:scvbounds}
 An MSPP with diagonal matrix $C=-\text{diag}(c_i)$ for $i=1,\cdots, p$  satisfies
$$
1 \leq c^2 \leq 2\,\frac{\kappa^2}{\gamma^2}-1,
$$ 
where $\kappa=\frac{\min{c_i}+\max{c_i}}{2}$ and $\gamma=\sqrt{(\min{c_i})(\max{c_i})}$.  
\end{proposition}
\begin{proof}
From Eq. \eqref{Eq:SCV-Interval}, we have  $c^2 + 1= 2(\boldsymbol{\pi} C {\mathbf 1} \boldsymbol{\pi} C^{-1} {\mathbf 1})$. 
 For $C=-\text{diag}(c_i)$, we have $C^{-1}=-\text{diag}(\frac{1}{c_i})$ and so,
 \[
 \frac{c^2 +1}{2}=\boldsymbol{\pi} C {\mathbf 1} \boldsymbol{\pi} C^{-1} {\mathbf 1} =\big(\sum_{i=1}^p\pi_i c_i\big)\big(\sum_{i=1}^p\pi_i \frac{1}{c_i}\big). 
\]
On the other hand, from the Cauchy-Schwarz inequality and the Kantorovich's Inequality (see \cite{steele2004cauchy}), we have:
$$
1=\Big[\sum_{i=1}^p\pi_i \,(c_i)^{\frac{1}{2}}\, (\frac{1}{c_i})^{\frac{1}{2}}\Big]^2 \leq \big(\sum_{i=1}^p\pi_i c_i\big)\big(\sum_{i=1}^p\pi_i \frac{1}{c_i}\big) \le \frac{\kappa^2}{\gamma^2}.
$$
%Now, consider that $m :=\min c_i$ and $M :=\max c_i$. Without loss of generality and by multiplying $m$, $M$, and all $c_i$, $i=1, \cdots, p$ by a positive constant, we can consider that $M=\frac{1}{m}$ (or $\gamma=1$). So, we have for $i=1, \cdots, p$: $c_i + \frac{1}{c_i} \leq m + \frac{1}{m} = 2 \kappa$ which results in (using the fact that $\sum_{i=1}^p \pi_i=1$):
%$$
%\Big(\sum_{i=1}^p \pi_i c_i\Big) + \Big(\sum_{i=1}^p \pi_i \frac{1}{c_i}\Big)= \sum_{i=1}^p \pi_i \big(c_i+\frac{1}{c_i}\big)\leq 2 \kappa.
%$$
%So, since the arithmetic mean of non-negative real valued numbers is always greater than their geometric mean, we have:
%$$
%\Big[ \Big(\sum_{i=1}^p \pi_i c_i\Big) \Big(\sum_{i=1}^p \pi_i \frac{1}{c_i}\Big)\Big]^{\frac{1}{2}}\leq \frac{\Big(\sum_{i=1}^p \pi_i c_i\Big)+ \Big(\sum_{i=1}^p \pi_i \frac{1}{c_i}\Big)}{2} \leq \kappa.
%$$
%Therefore, 
%$$
%\Big(\sum_{i=1}^p \pi_i c_i\Big) \Big(\sum_{i=1}^p \pi_i \frac{1}{c_i}\Big) \leq \kappa^2.
%$$ 
%The above inequality  is known as the Kantorovich's Inequality (see \cite{steele2004cauchy}) and is written for $\gamma \neq 1$ as:
%\begin{equation}\label{Cauchy2}
%\big(\sum_{i=1}^p\pi_i c_i\big)\big(\sum_{i=1}^p\pi_i \frac{1}{c_i}\big) \leq \frac{\kappa^2}{\gamma^2}.
%\end{equation}
Combination of the above two equations results in:  
$$
 1 \leq \frac{c^2 +1}{2} \leq \frac{\kappa^2}{\gamma^2},
$$
which completes the proof.

\end{proof}

%%%%%%%%%%%%%%%%
%%%%%%%%%%%%%%%%

However, for MMPPs while we believe the result is true (see numerical evidence in the next section) we don't have a general proof for properties (III) or (IV). Still, for two state MMPPs (MMPP$_2$) things are easier and we are able to show that all properties (I)-(IV) hold:

\begin{proposition}
\label{prop:mmpp2}
For a two-state MMPP$_2$, $c^2>1$ and $d^2>1$, $h(t)$ is DHR and the stochastic order $T_1^{\boldsymbol{\pi}} \ge_{\mbox{st}} T_1^{\boldsymbol{\alpha}}$ holds.
\end{proposition}
%%%%%%%%%%%%%%%
\begin{proof}
Consider an MMPP$_2$ with parameters  
\[
D=\left(\begin{array}{cc}
\lambda_1&0\\
0& \lambda_2
\end{array}\right)
\qquad
\mbox{and}
\qquad 
C=\left(\begin{array}{cc}
-\sigma_1-\lambda_1 & \sigma_1\\
\sigma_2 & -\sigma_2-\lambda_2
\end{array}\right).
\]
Then,  $\boldsymbol{\pi}=\frac{1}{\sigma_1+\sigma_2}(\sigma_2,\, \sigma_1)$. As in \cite{heffes1986markov}, evaluation of the transient deviation matrix through (for e.g.) Laplace transform inversion yields:
$$
\frac{\Var(N(t))}{\E[N(t)]}=1+\frac{2\sigma_1\sigma_2(\lambda_1-\lambda_2)^2}{(\sigma_1+\sigma_2)^2(\lambda_1\sigma_2+\lambda_2\sigma_1)}-\frac{2\sigma_1\sigma_2(\lambda_1-\lambda_2)^2}{(\sigma_1+\sigma_2)^3(\lambda_1\sigma_2+\lambda_2\sigma_1)t}(1-e^{-(\sigma_1+\sigma_2)t}).
$$
Therefore from~\eqref{eq:3535}, we have
$$
d^2=1+ \frac{2\sigma_1\sigma_2(\lambda_1-\lambda_2)^2}{(\sigma_1+\sigma_2)^2(\lambda_1\sigma_2+\lambda_2\sigma_1)}.
$$  

Further, explicit computation yields,
$$
c^2 = 1 + \frac{2\sigma_1\sigma_2(\lambda_1-\lambda_2)^2}{(\sigma_1+\sigma_2)^2(\lambda_2\sigma_1+\lambda_1(\lambda_2+\sigma_2))}.
$$
Thus it is evident that the MMPP$_2$ has $d^2 >1\, , c^2 > 1$
%strictly bursty 
as long as $\lambda_1 \neq \lambda_2$ and $d^2=c^2=1$ when $\lambda_1 = \lambda_2$. 

For DHR and the stochastic order, first we note that for an MMPP$_2$ with the above parameters, $\boldsymbol{\alpha}=\frac{1}{\sigma_1 \lambda_1+\sigma_2 \lambda_2}(\sigma_1 \lambda_1, \sigma_2 \lambda_2)$. 
By setting $B= \sigma_1 +\sigma_2 + \lambda_1+\lambda_2$ and $A=\sigma_2 \lambda_1 +\lambda_2(\sigma_1+\lambda_1)$, after some simplification, Eq.~\eqref{Eq:HDR}  is given by:
$$
- \frac{A e^{-Bt}\sigma_1 \sigma_2(\lambda_1-\lambda_2)^2}{(\sigma_2 \lambda_1+\sigma_1\lambda_2)^2},
$$
which is strictly negative for $\lambda_1 \neq \lambda_2$ and is zero for $\lambda_1 = \lambda_2$.
For the stochastic order, from Eq.~\eqref{Eq:SO}, we have:
$$
(\boldsymbol{\pi} - \boldsymbol{\alpha})e^{C t} \mathbf{1}= \frac{e^{-\frac{t}{2} \big(B+\sqrt{B^2-4A} \big)
}\big(-1+e^{t\sqrt{B^2-4A}}\big)\sigma_1 \sigma_2 (\lambda_1-\lambda_2)^2 }{(\sigma_1+\sigma_2) (\sigma_1 \lambda_2+\sigma_2 \lambda_1) \sqrt{B^2-4A}},
$$
which is strictly positive for $\lambda_1 \neq \lambda_2$ and is zero for $\lambda_1 = \lambda_2$.

\vspace{30pt}

\end{proof}

%%%%%%%%%%%%%%%%%%%%%%%%%%%%%%%%%%%%%%%%%%%%%%%%
%%%%%%%%%%%%%%%%%%%%%%%%%%%%%%%%%%%%%%%%%%%%%%%%
%%%%%%%%%%%%%%%%%%%%%%%%%%%%%%%%%%%%%%%%%%%%%%%%
%%%%%%%%%%%%%%%%%%%%%%%%%%%%%%%%%%%%%%%%%%%%%%%%
%%%%%%%%%%%%%%%%%%%%%%%%%%%%%%%%%%%%%%%%%%%%%%%%
%%%%%%%%%%%%%%%%%%%%%%%%%%%%%%%%%%%%%%%%%%%%%%%%
%%%%%%%%%%%%%%%%%%%%%%%%%%%%%%%%%%%%%%%%%%%%%%%%
%%%%%%%%%%%%%%%%%%%%%%%%%%%%%%%%%%%%%%%%%%%%%%%%
%%%%%%%%%%%%%%%%%%%%%%%%%%%%%%%%%%%%%%%%%%%%%%%%
\section{Conjectures for MMPP}
\label{sec4}

We embarked on this research due to the folklore assumption that for MMPP, $c^2 \ge 1$ (III). Initially we believed that it is easy to verify, however to date there isn't a known proof for an arbitrary irreducible MMPP. Still, we conjecture that both (III) and (IV) hold for MMPPs:

\begin{conjecture}
\label{conj:1}
For an irreducible MMPP, $c^2~\ge~1$.
\end{conjecture}

\begin{conjecture}
\label{conj:2}
For an  irreducible MMPP, $T_1^{\boldsymbol{\pi}} \ge_{\mbox{st}} T_1^{\boldsymbol{\alpha}}$.
\end{conjecture}

In an attempt to disprove these conjectures or alternatively gain confidence in their validity, we carried out an extensive numerical experiment. Our experiment works by generating random instances of MMPPs . Each instance is generated by first generating a matrix $Q$ with uniform$(0,1)$ off-diagonal entries and diagonal entries that ensure row sums are $0$. We then generate a matrix $D$ with diagonal elements that are exponentially distributed with rate $1$. Such a $(Q,D)$ pair then implies ${\boldsymbol{\pi}}$ and ${\boldsymbol{\alpha}}$. For each such MMPP we calculate $\boldsymbol{\pi} C {\mathbf 1} \boldsymbol{\pi} C^{-1} {\mathbf 1} -1 $ as in \eqref{Eq:c2} and $(\boldsymbol{\pi} - \boldsymbol{\alpha})e^{C t} \mathbf{1}$ as in \eqref{Eq:SO}, where we take $t \in \{0,0.2,0.4,\ldots,9.8,10.0\}$. We then ensure that both of these quantities are non-negative.

We repeated this experiment for $10^6$ random MMPP instances of orders $3,4,5$ and $6$. In all cases the calculated quantities were greater that $-10^{-15}$.  Note that in certain cases, the quantity associated with (IV) was negative and lying in the range $(-10^{-15},-10^{-16}]$. We attribute this to numerical error stemming from the calculation of the matrix exponential $e^{Ct}$. We ran our experiments with the Julia programming language, V1.0. The calculation time was about 1.5 hours.

This provides some evidence for the validity of Conjectures 1 and 2, although it is clearly not a proof. Further, we note that it is possible that some extreme cases exist that are not likely to come up by uniformly and randomly generating entries of $Q$. For example the cyclic matrix $Q$ in \eqref{Example}. For this we have also considered random cyclic $Q$ matrices with non-zero entries similar to \eqref{Example}. We generated $10^6$ such (order $4$) examples and all agreed with (III) and (IV).

%%%%%%%%%%%%%%%%%%%%%%%%%%%%%%%%%%%%%%%%%%%%%%%%
%%%%%%%%%%%%%%%%%%%%%%%%%%%%%%%%%%%%%%%%%%%%%%%%
%%%%%%%%%%%%%%%%%%%%%%%%%%%%%%%%%%%%%%%%%%%%%%%%
%%%%%%%%%%%%%%%%%%%%%%%%%%%%%%%%%%%%%%%%%%%%%%%%
%%%%%%%%%%%%%%%%%%%%%%%%%%%%%%%%%%%%%%%%%%%%%%%%
%%%%%%%%%%%%%%%%%%%%%%%%%%%%%%%%%%%%%%%%%%%%%%%%
%%%%%%%%%%%%%%%%%%%%%%%%%%%%%%%%%%%%%%%%%%%%%%%%
%%%%%%%%%%%%%%%%%%%%%%%%%%%%%%%%%%%%%%%%%%%%%%%%
%%%%%%%%%%%%%%%%%%%%%%%%%%%%%%%%%%%%%%%%%%%%%%%%
\section{Conclusion}
\label{sec5}

We have highlighted various related properties for point processes on the line and MAPs exhibiting diagonal matrices ($C$ or $D$) in particular. Showing that $c^2 \ge 1$ for MMPP and establishing the stochastic order $T_1^{\boldsymbol{\pi}} \ge_{\mbox{st}} T_1^{\boldsymbol{\alpha}}$ remains an open problem. We have shown this for MMPPs of order $2$ and using a similar technique to our MSPP proof, we can also show it for MMPPs with symmetric $C$ matrices. However for general MMPPs this remains an open problem.

We note, that stepping outside of the matrix analytic paradigm and considering general Cox processes is also an option. In fact, since any Cox process can be approximated by an MMPP, we believe that versions of conjectures $1$ and $2$ also hold for Cox processes under suitable regularity conditions.

There is also a related branch of questions dealing with characterizing the Poisson process via $c^2 = 1$ and considering when an MMPP is Poisson. For example, for the general class of MAPs, the authors of \cite{bean2000map} provide a condition for determining if a given MAP is Poisson. It is not hard to construct a MAP with $c^2 = 1$ that is not Poisson. But, we believe that all MMPPs with $c^2 = 1$ are Poisson. Yet, we don't have a proof. Further, we believe that for an MMPP, if $c^2 = 1$ then all $\lambda_i$ are equal (the converse is trivially
true). We don't have a proof of this either. Related questions also hold for the more general Cox processes.

We also note that the MSPP class of process that we considered generalized hyper-exponential renewal processes as well as a class of processes called Markovian Transition Counting Processes (MTCP) as in \cite{asanjarani2016queueing}.

%%%%%%%%%%%%%%%%%%%%%%%%%%%%%%
%%%%%%%%%%%%%%%%%%%%%%%%%%%%%%
%%%%%%%%%%%%%%%%%%%%%%%%%%%%%%

\section*{Acknowledgement}
Azam Asanjarani's research is supported by the Australian Research Council Centre of Excellence for the Mathematical and Statistical Frontiers (ACEMS). Yoni Nazarathy is supported by Australian Research Council Grant DP180101602. We thank Soren Asmussen, Qi-Ming He, Illes Horvath, Peter Taylor and Miklos Telek for useful discussions and insights related to this problem.

%%%%%%%%%%%%%%%%%%%%%%%%%%%%%%
%%%%%%%%%%%%%%%%%%%%%%%%%%%%%%
%%%%%%%%%%%%%%%%%%%%%%%%%%%%%%

\end{document}